\documentclass[11pt]{amsart}
\usepackage{amssymb}
\usepackage{hyperref}
\usepackage{latexsym}
\usepackage{amsthm}
\usepackage{tikz}
\usepackage{amscd}
\usepackage[all,cmtip]{xy}
\makeatletter
\@namedef{subjclassname@2010}{\textup{2010} Mathematics Subject Classification}
\makeatother

\newtheorem{theorem}{Theorem}[section]
\newtheorem{lemma}[theorem]{Lemma}
\theoremstyle{definition}

\newtheorem{example}[theorem]{Example}

\newtheorem{remark}[theorem]{Remark}
\newtheorem{proposition}[theorem]{Proposition}
\newtheorem{corollary}[theorem]{Corollary}
\newcommand*{\Ker}{\operatorname{Ker}}
\newcommand*{\Image}{\operatorname{Im}}
\newtheorem*{acknowledgement}{Acknowledgement}

\begin{document}
\title{Symmetric continuous cohomology of topological groups}
\author{Mahender Singh}
\email{mahender@iisermohali.ac.in}
\address{Indian Institute of Science Education and Research (IISER) Mohali, Phase 9, Sector 81, Knowledge City, S A S Nagar (Mohali), Post Office Manauli, Punjab 140306, India}

\begin{abstract}
In this paper, we introduce a symmetric continuous cohomology of topological groups. This is obtained by topologizing a recent construction due to Staic \cite{Staic1}, where a symmetric cohomology of abstract groups is constructed. We give a characterization of topological group extensions that correspond to elements of the second symmetric continuous cohomology. We also show that the symmetric continuous cohomology of a profinite group with coefficients in a discrete module is equal to the direct limit of the symmetric cohomology of finite groups. In the end, we also define symmetric smooth cohomology of Lie groups and prove similar results.
\end{abstract}

\maketitle

\section{Introduction}\label{section1}
The cohomology of abstract groups came into being with the fundamental work of Eilenberg and MacLane \cite{Eilenberg1, Eilenberg2}. The theory developed rapidly with the works of Eilenberg, MacLane, Hopf, Eckmann, Segal, Serre and many other authors. The cohomology of groups has been a popular research subject and has been studied from different perspectives with applications in algebraic number theory, algebraic topology and Lie algebras, to name a few. A detailed account of the history of the subject appears in \cite{Weibel}.

When the group under consideration is equipped with a topology, then it is natural to look for a cohomology theory which also takes the topology into account. This lead to many new cohomology theories of topological groups and the topology was first inserted in the formal definition of group cohomology in the works of Heller \cite{Heller}, Hu \cite{Hu} and van Est \cite{vanEst}.

In \cite{Fiedorowicz}, Fiedorowicz and Loday defined a homology theory of crossed simplicial groups. Motivated by their construction, Staic \cite{Staic1} introduced the notion of the $\Delta$-group $\Gamma(X)$ for a topological space $X$. Given a group $G$ and a $G$-module $A$, for each $n \geq 0$, Staic defined an action of the symmetric group $\Sigma_{n+1}$ on the standard $n$th cochain group $C^n(G, A)$ used to compute the usual group cohomology and proved it to be compatible with the standard coboundary operators $\partial^n$. Thus, the subcomplex $\{C^n(G, A)^{\Sigma_{n+1}},\partial^n \}_{n \geq 0}$ of invariant elements of $\{C^n(G, A),\partial^n \}_{n \geq 0}$ gives a new cohomology theory $HS^*(G, A)$ called the symmetric cohomology. Staic showed that the $\Delta$-group $\Gamma(X)$ is determined by the action of $\pi_1(X)$ on $\pi_2(X)$ and an element of $HS^3(\pi_1(X), \pi_2(X))$.

The inclusion of the cochain groups $C^n(G, A)^{\Sigma_{n+1}} \hookrightarrow C^n(G, A)$ induces a homomorphism from $HS^n(G,A) \to H^n(G,A)$. In \cite{Staic1}, it is shown that $HS^2(G,A) \to H^2(G,A)$ is injective.

It is well known that, if $A$ is a $G$-module, then there is a bijection between $H^2(G,A)$ and the set of equivalence classes of group extensions of $G$ by $A$ with the given $G$-module structure. Therefore, it is natural to ask what kind of group extensions correspond to elements of the second symmetric cohomology. In \cite{Staic2}, Staic proved that $HS^2(G,A)$ is in bijection with the set of equivalence classes of group extensions $0 \to A \to E \to G \to 1$ admitting a section $s: G \to E$ with the property that $s(g^{-1})=s(g)^{-1}$ for all $g \in G$. Note that the condition is slightly weaker than $s$ being a homomorphism. We shall see that there are examples of non-split extensions of groups admitting such a section.

The purpose of this paper is to topologize this construction and introduce a symmetric continuous cohomology of topological groups. As for the discrete case, we give a characterization of topological group extensions that correspond to elements of the second symmetric continuous cohomology. We also show that the symmetric continuous cohomology of a profinite group with coefficients in a discrete module is equal to the direct limit of the symmetric cohomology of finite groups. We similarly define symmetric smooth cohomology of Lie groups.

The paper is organized as follows. In Section \ref{section2}, we fix some notation and recall some known definitions and results that will be used in the paper. In Section \ref{section3}, we introduce the symmetric continuous cohomology of topological groups. In Section \ref{section4}, we give some examples to illustrate the proposed cohomology theory. In Section \ref{section5}, we prove some properties of the symmetric continuous cohomology of topological groups. In Section \ref{section6}, we discuss the symmetric continuous cohomology of profinite groups. Finally, in Section \ref{section7}, we define the symmetric smooth cohomology of Lie groups and prove some of its properties.

\begin{acknowledgement}
The author would like to thank the referee for comments which improved the presentation of the paper. The author is grateful to the MathOverflow community which was helpful in clarifying some examples. The author would also like to thank the Department of Science and Technology of India for support via the INSPIRE Scheme IFA-11MA-01/2011 and the SERC Fast Track Scheme SR/FTP/MS-027/2010.
\end{acknowledgement}
\bigskip

\section{Notation and terminology}\label{section2}
In this section, we fix some notation and recall some known definitions and results. We refer the reader to Brown \cite{Brown} for basic material on the cohomology of groups. For any extension $0 \to A \to E \to G \to 1$ of groups (abstract, topological or Lie), the group $A$ is written additively and the groups $E$ and $G$ are written multiplicatively, unless otherwise stated or it is clear from the context.

\subsection*{Cohomology of abstract groups} 
Let us recall the construction of the cochain complex defining the cohomology of abstract groups (groups without any other structure). Let $G$ be a group and $A$ be a $G$-module. More precisely, there is a group action $$G \times A \to A$$ by automorphisms. As usual $A$ is written additively and $G$ is written multiplicatively, unless otherwise stated or it is clear from the context. For each $n \geq 0$, the group of $n$-cochains $C^n(G, A)$ is the group of all maps $\sigma: G^n \to A$. The coboundary $$\partial^n:C^n(G, A) \to C^{n+1}(G, A)$$ is given by
\begin{equation}\label{eqn1}
\begin{split}
\partial^n(\sigma)(g_1,...,g_{n+1})& = g_1\sigma(g_2,...,g_{n+1}) \\
& +  \sum_{i=1}^n (-1)^{i+1}\sigma(g_1,...,g_ig_{i+1},...,g_{n+1}) \\
& + \sigma(g_1,...,g_n),
\end{split}
\end{equation}
for all $\sigma \in C^n(G, A)$ and $(g_1,...,g_{n+1}) \in G^{n+1}$. It is straightforward to verify that $\partial^{n+1}\partial^n=0$ and hence we obtain a cochain complex. Let $Z^n(G, A)= \Ker (\partial^n)$ be the group of $n$-cocycles and $B^n(G, A)= \Image (\partial^{n-1})$ be the group of $n$-coboundaries. Then the $n$th cohomology group is given by $$H^n(G,A)= Z^n(G, A)/B^n(G, A).$$ If $\sigma \in Z^n(G, A)$ is a $n$-cocyle, we denote by $[\sigma] \in H^n(G, A)$ the corresponding cohomology class.

\subsection*{Symmetric cohomology of abstract groups}
For each $n \geq 0$, let $\Sigma_{n+1}$ be the symmetric group on $n+1$ symbols. In \cite{Staic1}, Staic defined an action of the symmetric group $\Sigma_{n+1}$ on the $n$th cochain group $C^n(G, A)$. Since the transpositions of adjacent elements form a generating set for $\Sigma_{n+1}$, it is enough to define the action of these transpositions $\tau_i = (i, i+1)$ for $1 \leq i \leq n$. For $\sigma \in C^n(G,A)$ and $(g_1,...,g_n) \in G^n$, define
\begin{equation}\label{eqn2}
\begin{split}
(\tau_1 \sigma)(g_1,g_2,g_3,...,g_n) &= -g_1 \sigma\big((g_1)^{-1},g_1g_2,g_3,...,g_n\big),\\
(\tau_i \sigma)(g_1,g_2,g_3,...,g_n) &= -\sigma\big(g_1,...,g_{i-2},g_{i-1}g_i, (g_i)^{-1}, g_ig_{i+1}, g_{i+2},...,g_n\big)~  \textrm{for}~ 1 < i < n,\\
(\tau_n \sigma)(g_1,g_2,g_3,...,g_n) &= -\sigma\big(g_1,g_2,g_3,...,g_{n-1}g_n, (g_n)^{-1}\big).
\end{split}
\end{equation}

It is shown in \cite{Staic1} that the above action is compatible with the coboundary operators $\partial^n$ and hence yields the subcomplex of invariants $$\{CS^n(G, A), \partial^n\}_{n \geq 0}= \{C^n(G, A)^{\Sigma_{n+1}}, \partial^n\}_{n \geq 0}.$$ The cohomology of this cochain complex is called the symmetric cohomology of $G$ with coefficients in $A$ and is denoted $HS^n(G,A)$. The cocycles and the coboundaries are called the symmetric cocycles and the symmetric coboundaries, respectively.

\subsection*{Continuous cohomology of topological groups}
We assume that all topological groups under consideration satisfy the $T_0$ separation axiom. Let $G$ be a topological group and $A$ be an abelian topological group. We say that $A$ is a topological $G$-module if there is a continuous action of $G$ on $A$. The continuous cohomology of topological groups was defined independently by Hu \cite{Hu}, van Est \cite{vanEst} and Heller \cite{Heller} as follows.

For each $n \geq 0$, let $C_{c}^n(G,A)$ be the group of all continuous maps from $G^n \to A$, where $G^n$ is the product topological group. The coboundary maps given by the standard formula as in \eqref{eqn1}, gives the cochain complex $\{C_{c}^n(G,A), \partial^n \}_{n \geq 0}$. The continuous cohomology of $G$ with coefficients in $A$ is defined to be the cohomology of this cochain complex and is denoted by $H_{c}^*(G,A)$.

Clearly, this cohomology theory coincides with the abstract cohomology theory when the groups under consideration are discrete (in particular finite). The low dimensional cohomology groups are as expected. More precisely, $H_{c}^0(G,A)=A^G$ and $Z_{c}^1(G,A)=$ the group of continuous crossed homomorphisms from $G$ to $A$.

An extension of topological groups $$0 \to A \stackrel{i}{\to} E \stackrel{\pi}{\to} G \to 1$$ is an algebraically exact sequence of topological groups with the additional property that $i$ is closed continuous and $\pi$ is open continuous. Note that, if we assume that $i$ and $\pi$ are only continuous, then $A$ viewed as a subgroup of $E$ may not have the relative topology and the isomorphism $E/i(A) \cong G$ may not be a homeomorphism.

A section to the given extension is a map $s:G \to E$ such that $\pi s(g)=g$ for all $g \in G$. Since $A$ is closed and $i$ is closed continuous, it follows that $i(A)= \pi^{-1}(\{1\})$ is a closed subgroup of $E$ and $i:A \to E$ is an embedding of $A$ onto a closed subgroup of $E$.

Let $G$ be a topological group and $A$ a topological $G$-module. Let $0 \to A \stackrel{i}{\to} E \stackrel{\pi}{\to} G \to 1$ be a topological group extension and let $s:G \to E$ be a section to $\pi$. Since $A$ is abelian, for $a \in A$ and $g \in G$, one can see that the element $i^{-1}\big( s(g)i(a)s(g)^{-1} \big)$ does not depend on the choice of the section. The extension $0 \to A \stackrel{i}{\to} E \stackrel{\pi}{\to} G \to 1$ is said to correspond to the given way in which $G$ acts on $A$ if $$ga=i^{-1}\big( s(g)i(a)s(g)^{-1} \big)~\textrm{for all}~a \in A~\textrm{and}~g \in G.$$

Consider the set of all topological group extensions of $G$ by $A$ corresponding to the given way in which $G$ acts on $A$. Two such extensions $0 \to A \stackrel{i}{\to} E \stackrel{\pi}{\to} G \to 1$ and $0 \to A \stackrel{i'}{\to} E' \stackrel{\pi'}{\to} G \to 1 $ are said to be equivalent if there exists an open continuous isomorphism $\phi:E \to E'$ such that the following diagram commute
$$
\xymatrix{
0 \ar[r] & A \ar@{=}[d] \ar[r]^{i} & E \ar[r]^{\pi} \ar[d]^{\phi} & G \ar[r] \ar@{=}[d] & 1 \\
0 \ar[r] & A \ar[r]^{i'} & E' \ar[r]^{\pi'} & G \ar[r] & 1.}
$$

For brevity, $E\cong E'$ denotes the equivalence of extensions. Heller \cite{Heller} and Hu \cite[Theorem 5.3]{Hu} independently proved the following result.

\begin{theorem}\label{hellerhu}
Let $G$ be a topological group and $A$ a topological $G$-module. Then $H_{c}^2(G,A)$ is in bijection with the set of equivalence classes of topological group extensions of $G$ by $A$ admitting a (global) continuous section and the given $G$-module structure.
\end{theorem}

We shall prove similar theorems using symmetric continuous cohomology of topological groups (Theorem \ref{theorem3.3}) and symmetric smooth cohomology of Lie groups (Theorem \ref{theorem7.2}) in the following sections.

An extension of topological groups is said to be topologically split if $E$ is $A \times G$ as a topological space. Note that if an extension of topological groups admits a continuous section, then the extension is topologically split. Extensions of topological groups admitting a continuous section are assured by the following theorem.

\begin{theorem}\cite[Theorem 2]{Shtern}
Let $G$ be a connected locally compact group. Then any topological group extension of $G$ by a simply connected Lie group admits a continuous
section.
\end{theorem}
\bigskip

\section{Symmetric continuous cohomology of topological groups}\label{section3}
In this section, we define the symmetric continuous cohomology of topological groups, having the expected cohomology groups in low dimension. From now on, $G$ is a topological group and $A$ a topological $G$-module. Let $n \geq 0$. Since $G$ is a topological group, for $(g_1,...,g_n) \in G^n$,
\begin{equation}\label{eqn3}
\begin{split}
(g_1,g_2,g_3,...,g_n) &\mapsto \big((g_1)^{-1},g_1g_2,g_3,...,g_n\big),\\
(g_1,g_2,g_3,...,g_n) &\mapsto \big(g_1,...,g_{i-2},g_{i-1}g_i, (g_i)^{-1}, g_ig_{i+1}, g_{i+2},...,g_n\big)~  \textrm{for}~ 1 < i < n,\\
(g_1,g_2,g_3,...,g_n) &\mapsto \big(g_1,g_2,g_3,...,g_{n-1}g_n, (g_n)^{-1}\big),
\end{split}
\end{equation}
are all continuous maps $G^n \to G^n$.

The continuity of the action $G \times A \to A$ and the maps given by \eqref{eqn3} shows that $\tau \sigma \in  C_{c}^n(G,A)$ for each $\tau \in \Sigma_{n+1}$ and $\sigma \in C_{c}^n(G,A)$. By \cite[Proposition 5.1]{Staic1}, the formulas given by \eqref{eqn2} define an action compatible with the coboundary operators $\partial^n$ given by \eqref{eqn1}.

This gives the subcomplex of invariants $$\{CS_{c}^n(G, A), \partial^n\}_{n \geq 0} =\{C_{c}^n(G, A)^{\Sigma_{n+1}}, \partial^n\}_{n \geq 0}.$$ We call the cohomology of this cochain complex the symmetric continuous cohomology of $G$ with coefficients in $A$ and denote it by $HS_{c}^n(G,A)$.

Clearly, when the groups under consideration are discrete, then $HS_{c}^n(G,A)=HS^n(G,A)$. When $G$ is a connected topological group and $A$ is a discrete G-module, then $G$ acts trivially on $A$ and the continuous cochains are only constant maps. Then it follows that $HS_{c}^0(G,A)=A$ and $HS_{c}^n(G,A)=0$ for each $n \geq 1$.

Observe that a 1-cocycle $\lambda:G \to  A$ is symmetric if $$\lambda(g)=-g\lambda(g^{-1})~\textrm{for all}~g \in G$$ and a 2-cocycle $\sigma:G\times G \to A$ is symmetric if
\begin{equation}\label{eqn4}
\begin{split}
\sigma(g,k)=-g\sigma(g^{-1}, gk) & = -\sigma(gk, k^{-1})~\textrm{for all}~g,k \in G.
\end{split}
\end{equation}

It is easy to establish the following properties.

\begin{proposition}\label{proposition3.1}
Let $G$ be a topological group and $A$ be a topological $G$-module. Then
\begin{enumerate}
\item $HS_{c}^0(G,A)=A^G= H_{c}^0(G, A)$
\item $ZS_{c}^1(G,A)=$ the group of symmetric continuous crossed homomorphisms from $G$ to $A$.
\end{enumerate}
\end{proposition}

\begin{proof} (1) is straightforward. By definition $$ZS_{c}^1(G,A)= \Ker \{\partial^1: CS_{c}^1(G,A) \to CS_{c}^2(G,A) \}.$$ Therefore, $\lambda \in ZS_{c}^1(G,A)$ if and only if $\lambda$ is continuous and satisfy $\lambda(g)=-g\lambda(g^{-1})$ and $\lambda (gk)=g\lambda(k) + \lambda (g)$ for all $g, k \in G$. In other words, $\lambda$ is a symmetric continuous crossed homomorphism. This proves (2).
\end{proof}

As in the discrete case, the inclusion of the subcomplex $$CS_{c}^*(G, A) \hookrightarrow C_{c}^*(G, A)$$ induces a homomorphism $$h^*:HS_{c}^*(G, A) \to H_{c}^*(G, A).$$ Clearly $h^*:HS_{c}^1(G, A) \to H_{c}^1(G, A)$ is injective. In dimension two, we have the following proposition, which is essentially \cite[Lemma 3.1]{Staic2}. We provide a proof here for the sake of completeness.

\begin{proposition}\label{proposition3.2}
The map $h^*:HS_{c}^2(G, A) \to H_{c}^2(G, A)$ is injective.
\end{proposition}
\begin{proof}
Let $\sigma$ represent an element in $\Ker(h^*)$. In other words, $\sigma \in ZS_{c}^2(G, A)\cap B_{c}^2(G, A)$. This implies that $\sigma$ is symmetric and there exists $\lambda \in C_{c}^1(G, A)$ such that $$\sigma (g,k)= \partial^1 \lambda(g,k)= g\lambda(k)-\lambda(gk) + \lambda (g)$$ for all $g, k \in G$. The symmetry of $\sigma$ gives,
$$\sigma(g,k)=-g\sigma(g^{-1}, gk)=-\lambda(gk) + g\lambda(k) -g\lambda(g^{-1})~\textrm{and}$$
$$\sigma(g,k)=-\sigma(gk, k^{-1})=-gk\lambda(k^{-1}) + \lambda(g) - \lambda(gk)~\textrm{for all}~g,k \in G.$$
By taking $g=1$ and equating the above two equations, we get $\lambda(k)= -k\lambda(k^{-1})$ for all $k \in G$. This shows that $\lambda \in CS_{c}^1(G, A)$ and hence $\sigma \in BS_{c}^2(G, A)$. Thus $\sigma$ represents the trivial element in $HS_{c}^2(G, A)$ and the map $h^*$ is injective.
\end{proof}

Note that the map $h^*$ need not be surjective in general. This will be illustrated by examples in Section \ref{section4}.

By Theorem \ref{hellerhu}, $H_{c}^2(G,A)$ classifies equivalence classes of topological group extensions of $G$ by $A$ admitting a continuous section.  In view of Proposition \ref{proposition3.2}, we would like to know which of these extensions correspond to $HS_{c}^2(G, A)$. Let $0 \to A \to E \to G \to 1$ be an extension of topological groups. We say that a section $s:G \to E$ is symmetric if $$s(g^{-1})=s(g)^{-1}~ \textrm{for all}~ g \in G.$$ For simplicity, we assume that $s$ satisfies the normalization condition $s(1)=1$. Let $\mathcal{C}(G,A)$ denote the set of equivalence classes of topological group extensions of $G$ by $A$ admitting a symmetric continuous section and being equipped with the given $G$-module structure. With these definitions, we prove the following theorem.

\begin{theorem}\label{theorem3.3}
Let $G$ be a topological group and $A$ be a topological $G$-module. Then there is a bijection $\Phi:\mathcal{C}(G,A) \to HS_{c}^2(G,A)$.
\end{theorem}
\begin{proof}
Let $0 \to A \stackrel{i}{\longrightarrow} E \stackrel{\pi}{\longrightarrow} G \to 1$ be a topological group extension of $G$ by $A$ admitting a symmetric continuous section $s: G \to E$ and corresponding to the given way in which $G$ acts on $A$. Every element of $E$ can be written uniquely as $i(a)s(g)$ for some $a \in A$ and $g \in G$. Since the extension corresponds to the given way in which $G$ acts on $A$, we have that $$ga = i^{-1}\big( s(g)i(a)s(g)^{-1} \big) ~\textrm{for all}~  a \in A ~\textrm{and}~ g \in G.$$ As $i$ and $s$ are continuous, we see that the action is continuous. Also, we have $$\pi(s(gh))= gh = \pi(s(g))\pi(s(h))= \pi(s(g)s(h))~\textrm{for all}~ g, h \in G.$$ Thus, there exists a unique element (say) $\sigma(g,h)$ in $A$ such that $$\sigma(g,h)= i^{-1} \big(s(g)s(h)s(gh)^{-1} \big).$$ Observe that $\sigma:G \times G \to A$ satisfies the condition 
\begin{equation}\label{eqn5}
\begin{split}
 g\sigma(h, k)- \sigma(gh, k) + \sigma(g, hk)- \sigma(g, h) = 0 & ~ \textrm{for all}~ g, h, k \in G.
\end{split}
\end{equation}
In other words, $\sigma$ is a 2-cocycle. Moreover, continuity of $i$ and $s$ implies that $\sigma$ is continuous. Finally, using the symmetry of $s$ and the action of $G$ on $A$, we show that $\sigma$ is in fact symmetric. That is, for all $g,h \in G$, we have
\begin{eqnarray}
-g\sigma(g^{-1}, gh) & = & -gi^{-1}\big(s(g^{-1})s(gh)s(h)^{-1}\big) \nonumber\\
& = & -gi^{-1} \big(s(g)^{-1}s(gh)s(h)^{-1}\big)\nonumber\\
& = & -i^{-1}\big(s(g)s(g)^{-1}s(gh)s(h)^{-1}s(g)^{-1} \big)\nonumber\\
& = & -i^{-1}\big(\big(s(g)s(h)s(gh)^{-1}\big)^{-1}\big)\nonumber\\
& = & i^{-1}\big( s(g)s(h)s(gh)^{-1} \big)\nonumber\\
& = & \sigma(g,h)\nonumber\\
\textrm{and}~-\sigma(gh, h^{-1}) & = & -i^{-1}\big( s(gh)s(h^{-1})s(g)^{-1}\big) \nonumber\\
& = & -i^{-1}\big(s(gh)s(h)^{-1}s(g)^{-1} \big) \nonumber\\
& = & -i^{-1}\big(\big(s(g)s(h)s(gh)^{-1}\big)^{-1}\big) \nonumber\\
& = & i^{-1}\big(s(g)s(h)s(gh)^{-1}\big)\nonumber\\
& = & \sigma(g,h).\nonumber
\end{eqnarray}
Thus, $\sigma$ gives an element in $HS_{c}^2(G,A)$.

We need to show that the cohomology class of $\sigma$ is independent of the choice of the section. Let $s, t:G \to E$ be two symmetric continuous sections. As above we get symmetric continuous 2-cocycles $\sigma,\mu:G \times G \to A$ such that for all $g, h \in G$, we have $$\sigma(g,h)= i^{-1} \big(s(g)s(h)s(gh)^{-1} \big)$$ and $$\mu(g,h)= i^{-1} \big(t(g)t(h)t(gh)^{-1} \big).$$ Since $s(g)$ and $t(g)$ satisfy $\pi(s(g))=g=\pi(t(g))$, there exists a unique element (say) $\lambda(g) \in A$ such that $$\lambda(g)= i^{-1} \big(s(g)t(g)^{-1} \big).$$ This yields a 1-cochain $\lambda: G \to A$ which is continuous and symmetric, as
\begin{eqnarray}
\tau_1 \lambda(g) & = & -g \lambda(g^{-1}) \nonumber\\
& = & -gi^{-1}\big(s(g^{-1})t(g^{-1})^{-1} \big) \nonumber\\
& = & -gi^{-1}\big(s(g)^{-1}t(g)\big) \nonumber\\
& = & -i^{-1}\big(s(g)s(g)^{-1}t(g)s(g)^{-1}\big) \nonumber\\
& = & -i^{-1}\big((s(g)t(g)^{-1})^{-1}\big) \nonumber\\
& = & i^{-1}\big(s(g)t(g)^{-1}\big) \nonumber\\
& = & \lambda(g)\nonumber.
\end{eqnarray}
Thus, we have that $$\sigma(g,h)-\mu(g,h)= g\lambda(h)-\lambda(gh)+ \lambda(g).$$ In other words, $\sigma-\mu \in BS_{c}^2(G, A)$ and hence $[\sigma]=[\mu]$ in $HS_{c}^2(G, A)$.

Let $0 \to A \stackrel{i'}{\to} E' \stackrel{\pi'}{\to} G \to 1 $ be an extension equivalent to $0 \to A \stackrel{i}{\to} E \stackrel{\pi}{\to} G \to 1$ via an open continuous isomorphism $\phi:E \to E'$. Then $s'= \phi s:G \to E'$ is a symmetric continuous section. It is clear that the 2-cocycle corresponding to $s'$ is same as the one corresponding to $s$. Hence equivalent extensions gives the same element in $HS_{c}^2(G,A)$.

Now we can define $$\Phi: \mathcal{C}(G,A) \to HS_{c}^2(G,A)$$ by mapping an equivalence class of extensions to the corresponding cohomology class as obtained above. The above arguments show that $\Phi$ is well defined.

We first prove that $\Phi$ is surjective. Let $\sigma \in ZS_{c}^2(G,A)$ be a symmetric continuous 2-cocyle representing an element in $HS_{c}^2(G,A)$. The symmetry of the 2-cocyle gives the equation \eqref{eqn4}. Let 
$$E_{\sigma}:=A \times G$$ 
be equipped with the product topology. Define a binary operation on $E_{\sigma}$ by 
$$(a, g)(b, h)=\big(a + gb + \sigma(g, h), gh\big)~\textrm{for all}~a,b \in A~\textrm{and}~g,h \in G.$$
It is routine to check that this binary operation gives a group structure on $E_{\sigma}$. Since $A$ is a topological $G$-module and the 2-cocyle $\sigma$ is continuous, the group operation from $E_{\sigma} \times E_{\sigma} \to E_{\sigma}$ and the inverting operation from $E_{\sigma} \to E_{\sigma}$ are continuous with respect to the product topology on $E_{\sigma}$. Hence, $E_{\sigma}$ is a topological group.

Clearly, the map $\pi:E_{\sigma} \to G$ given by $\pi(a,g)=g$ is an open continuous homomorphism; and the map $i:A \to E_{\sigma}$ given by $i(a)= (a,1)$ is an embedding of $A$ onto the closed subgroup $i(A)$ of $E_{\sigma}$. This gives the following extension of topological groups $$0 \to A \stackrel{i}{\longrightarrow} E_{\sigma} \stackrel{\pi}{\longrightarrow} G \to 1.$$ The extension has an obvious continuous section $s:G \to E_{\sigma}$, given by $s(g)=(0,g)$. Using the group operation and the symmetry of $\sigma$, we get
$$s(g)\big(s(g^{-1})s(gh)\big) = \big(g\sigma(g^{-1}, gh) + \sigma(g, h), gh\big) = (0, gh) = s(gh)~ \textrm{and}$$
$$\big(s(gh)s(h^{-1}) \big)s(h) = \big(\sigma(g, h) + \sigma(gh, h^{-1}), gh\big) = (0, gh) = s(gh).$$
This gives $s(g^{-1})=s(g)^{-1}$ and hence the section $s$ is symmetric.

Note that by \eqref{eqn5} and the normalization of the section, we have $$\sigma(g,1)= \sigma(1,g)= \sigma(1,1)=0~\textrm{for all}~g \in G.$$ For all $a \in A$ and $g \in G$, we have
\begin{eqnarray}
i^{-1}\big( s(g)i(a)s(g)^{-1} \big) & = &  i^{-1}\big( (0,g)(a,1)(0,g)^{-1} \big)\nonumber\\
& = &  i^{-1}\big( (ga + \sigma(g,1),g)(0,g)^{-1} \big)\nonumber\\
& = & i^{-1}\big( (ga + \sigma(g,1),g)(-g^{-1}\sigma(g,g^{-1}), g^{-1}) \big)\nonumber\\
& = & i^{-1}\big( (ga + \sigma(g,1) + g(-g^{-1}\sigma(g,g^{-1})) + \sigma(g,g^{-1}), 1) \big)\nonumber\\
& = & i^{-1}\big( (ga + \sigma(g,1), 1) \big)\nonumber\\
& = & i^{-1}\big( (ga, 1) \big)\nonumber\\
& = & ga.\nonumber
\end{eqnarray}
Thus, the extension $0 \to A \to E_{\sigma} \to G \to 1$ corresponds to the given $G$-module structure on $A$.

Next, for all $g, h \in G$, we have
\begin{eqnarray}
\sigma(g,h) & = &  i^{-1}\big((\sigma(g,h),1) \big)\nonumber\\
& = &  i^{-1}\big( (\sigma(g,h)-\sigma(1,gh),1) \big)\nonumber\\
& = &  i^{-1}\big( (\sigma(g,h)-\sigma(1,gh),1)(0,gh)(0,gh)^{-1} \big)\nonumber\\
& = &  i^{-1}\big( (\sigma(g,h),gh)(0,gh)^{-1} \big)\nonumber\\
& = &  i^{-1}\big( (0,g)(0,h)(0,gh)^{-1} \big)\nonumber\\
& = &  i^{-1}\big( s(g)s(h)s(gh)^{-1} \big).\nonumber
\end{eqnarray}
Thus, $\sigma$ is the 2-cocycle corresponding to the section $s$. Hence $\Phi$ is surjective.

Finally, we prove that $\Phi$ is injective. Let $0 \to A \stackrel{i}{\to} E \stackrel{\pi}{\to} G \to 1$ and $0 \to A \stackrel{i'}{\to} E' \stackrel{\pi'}{\to} G \to 1$ be two topological group extensions admitting symmetric continuous sections $s$ and $s'$, respectively. Let $\sigma$ and $\sigma'$ be the 2-cocyles associated to $s$ and $s'$, respectively. Suppose that $\sigma$ and $\sigma'$ represent the same element in $HS_{c}^2(G,A)$. In other words, $\sigma'- \sigma = \partial^1(\lambda)$ for some $\lambda \in CS_{c}^1(G,A)$. Define $t:G \to E$ by $$t(g)= i\lambda(g)s(g)~\textrm{for all}~ g \in G.$$ We can see that $t$ is a continuous section to $\pi$ and also gives rise to the 2-cocycle $\sigma'$ as
\begin{eqnarray}
i^{-1}\big( t(g)t(h)t(gh)^{-1} \big) & = & i^{-1}\big(i\lambda(g)s(g)i\lambda(h)s(h) s(gh)^{-1} i\lambda(gh)^{-1} \big)\nonumber\\
& = &  i^{-1}\big( i\lambda(g)s(g)i\lambda(h)s(g)^{-1}s(g)s(h) s(gh)^{-1} i\lambda(gh)^{-1}\big)\nonumber\\
& = &  i^{-1}\big( i\lambda(g)s(g)i\lambda(h)s(g)^{-1}i(\sigma(g,h)) i\lambda(gh)^{-1}\big)\nonumber\\
& = &  \lambda(g) + i^{-1}\big( s(g)i\lambda(h)s(g)^{-1} \big) + \sigma(g,h) - \lambda(gh)\nonumber\\
& = &  \lambda(g) + g \lambda(h)+ \sigma(g,h) -\lambda(gh)\nonumber\\
& = &  \sigma'(g,h).\nonumber
\end{eqnarray}

Let $0 \to A \to E_{\sigma'} \to G \to 1$ be the extension associated to $\sigma'$. Define $\phi_t:E_{\sigma'} \to E$ by $$\phi_t(a,g)= i(a)t(g)~\textrm{for all}~a \in A~ \textrm{and}~ g \in G.$$ Clearly, $\phi_t$ is continuous; and it is a homomorphism because
\begin{eqnarray}
\phi_t\big( (a,g)(b,h)\big) & = & \phi_t(a+gb+\sigma'(g,h), gh)\nonumber\\
& = &  i(a+gb+\sigma'(g,h))t(gh)\nonumber\\
& = &  i(a)i(gb)i(\sigma'(g,h))t(gh)\nonumber\\
& = &  i(a)t(g)i(b)t(g)^{-1}t(g)t(h)t(gh)^{-1}t(gh)\nonumber\\
& = &  i(a)t(g)i(b)t(h)\nonumber\\
& = &  \phi_t(a,g) \phi_t(b,h).\nonumber
\end{eqnarray}
It is easy to see that $\phi_t$ is bijective with inverse $i(a)t(g) \mapsto (a,g)$. As both $E_{\sigma'}$ and $E$ have the product topology, the inverse homomorphism is also continuous and hence $\phi_t$ is an equivalence of extensions $E_{\sigma'} \cong E$. Similarly, define $\phi_{s'}:E_{\sigma'} \to E'$ by $$\phi_{s'}(a,g)= i'(a)s'(g)~\textrm{for all}~a \in A~ \textrm{and}~ g \in G.$$ Just as above, we can see that $\phi_{s'}$ is an equivalence of extensions $E_{\sigma'} \cong E'$. Hence we get an equivalence of extensions $E \cong E'$ proving that $\Phi$ is injective. The proof of the theorem is now complete.
\end{proof}

We conclude this section with the following remarks.

\begin{remark}
Recall that a topological group $G$ is said to be a free topological group if there exists a completely regular space $X$ such that:
\begin{enumerate}
\item [(i)] $X$ is topologically embeddable in $G$;
\item [(ii)] when embedded $X$ generates $G$;
\item [(iii)] every continuous map from $X$ to a topological group $H$ can be extended to a unique continuous homomorphism from $G$ to $H$.
\end{enumerate}

If $G$ is a free topological group, then $H_{c}^2(G, A)=0$ by \cite[Proposition 5.6]{Hu}. Thus, if $G$ is a free topological group and $A$ is a topological $G$-module, then $HS_{c}^2(G, A)=0$ by Proposition \ref{proposition3.2}.
\end{remark}

\begin{remark}
Let $G$ be a topological group and $A$ be a topological $G$-module. Restriction to the underlying abstract group structure gives, for each $n \geq 0$, the homomorphisms (with the same notation) $$r^*:H^n_{c}(G,A) \to H^n(G,A)$$ and $$r^*:HS^n_{c}(G,A) \to HS^n(G,A).$$ It is easy to see that the following diagram is commutative
$$
\xymatrix{
HS_{c}^2(G, A) \ar[d]^{h^*} \ar[r]^{r^*} & HS^2(G, A) \ar[d]^{h^*}\\
H_{c}^2(G, A) \ar[r]^{r^*} & H^2(G, A).}
$$
Note that both the vertical maps are injective by Proposition \ref{proposition3.2} and \cite[Lemma 3.1]{Staic2}. When $G$ and $A$ are locally compact groups, then Moore \cite{Moore2} defined a cohomology theory $H_{m}^*(G, A)$ with measurable cochains. He showed that if $G$ is perfect, then the restriction map $$H_{m}^2(G, A) \to  H^2(G, A)$$ is injective \cite[Theorem 2.3]{Moore2}. Further, he showed that, if $G$ is a profinite group and $A$ a discrete $G$-module or $G$ is a Lie group and $A$ a finite dimensional $G$-vector space, then the restriction map $$H_{c}^2(G, A) \stackrel{\cong}{\to} H_{m}^2(G, A)$$ is an isomorphism \cite[p.32]{Moore1}. When the group $G$ is perfect, combining the above two results of Moore shows that the restriction map $$r^*:H_{c}^2(G, A) \to H^2(G, A)$$ is injective. This together with the injectivity of the vertical maps $h^*$ in the above commutative diagram proves the following: Let $G$ be a perfect group satisfying either of the following conditions:
\begin{enumerate}
\item[(i)] $G$ is a profinite group and $A$ a discrete $G$-module;
\item[(ii)] $G$ is a Lie group and $A$ a finite dimensional $G$-vector space.
\end{enumerate}
Then the restriction map $r^*:HS_{c}^2(G, A) \to HS^2(G, A)$ is injective. We shall show in Example \ref{example4.4} that this map is not injective in general.
\end{remark}

\section{Some examples}\label{section4}
In this section, we give some examples to illustrate the symmetric continuous cohomology introduced in the previous section.

\begin{example}\label{example4.1}
We give an example of a topological group and a topological module for which the two cohomology theories $HS_c^*(-,-)$ and $H_c^*(-,-)$ are different. More precisely, we show that, the map $h^*:HS_{c}^2(-,-) \to H_{c}^2(-,-)$ is not surjective in general.

First consider the extension $$0 \to \mathbb{Z} \stackrel{i}{\to} \mathbb{Z} \times \mathbb{Z}/2 \stackrel{\pi}{\to} \mathbb{Z}/4 \to 0,$$ where $i(n)= (2n, \overline{n})$ and $\pi(n, \overline{m})= \overline{n + 2m}$. Here $\overline{n}$ denotes the class of $n$ modulo 2 or 4 depending on the context. Let $s:\mathbb{Z}/4 \to \mathbb{Z} \times \mathbb{Z}/2$ be the section given by $$s(\overline{0})= (0,\overline{0}),~ s(\overline{1})= (-1,\overline{1}),~ s(\overline{2})=(0,\overline{1})~ \textrm{and}~ s(\overline{3})=(1,\overline{1}).$$ Equipping each group with the discrete topology, we can consider this as an extension of topological groups. Then $s$ is clearly a symmetic continuous section. Let $G$ be a non-discrete abelian topological group. Consider the extension 
\begin{equation}\label{eqn6}
\begin{split}
0 \to \mathbb{Z} \times G \stackrel{i'}{\to} \mathbb{Z} \times \mathbb{Z}/2 \times G \times G \stackrel{\pi'}{\to} \mathbb{Z}/4\times G \to 0 &,
\end{split}
\end{equation}
where $i'(n, g)= (i(n), g, 0)$ and $\pi'(n, \overline{m}, g, h)= (\pi(n, \overline{m}), h)$. This is a non-split extension of topological groups. Note that the extension also admits a symmetric continuous section $s':\mathbb{Z}/4 \times G \to \mathbb{Z}\times \mathbb{Z}/2 \times G \times G$ given by $$s'(\overline{n}, h)= (s(\overline{n}), 0, h).$$ Therefore, \eqref{eqn6} represents a unique  non-trivial element in $HS_{c}^2(\mathbb{Z}/4 \times G,  \mathbb{Z} \times G)$.

Next consider the extension $$0 \to \mathbb{Z} \stackrel{j}{\to} \mathbb{Z} \stackrel{\nu}{\to} \mathbb{Z}/4 \to 0,$$ where $j(n)= 4n$ and $\nu(n)= \overline{n}$. With the discrete topology, we can consider this as an extension of topological groups. This extension does not admit any symmetric continuous section. For any non-discrete abelian topological group $G$, we get the following non-split extension of topological groups
\begin{equation}\label{eqn7}
\begin{split}
0 \to \mathbb{Z} \times G \stackrel{j'}{\to} \mathbb{Z} \times G \times G \stackrel{\nu'}{\to} \mathbb{Z}/4\times G \to 0 &,
\end{split}
\end{equation}
where $j'(n, g)= (j(n), g, 0)$ and $\nu'(n, g, h)= (\nu(n), h)$. Let $s:\mathbb{Z}/4 \to \mathbb{Z}$ be any continuous section, which exists as the topologies are discrete. Then the section $s':\mathbb{Z}/4 \times G \to \mathbb{Z}\times G \times G$ given by $$s'(\overline{n}, h)= (s(\overline{n}), 0, h)$$ is continuous. However, there does not exist any symmetric continuous section. Therefore, \eqref{eqn7} represents a unique non-trivial element in $H_{c}^2(\mathbb{Z}/4 \times G,  \mathbb{Z} \times G)$, but does not represent an element in $HS_{c}^2(\mathbb{Z}/4 \times G,  \mathbb{Z} \times G)$.
\end{example}

\begin{example}\label{example4.2}
We now give an example which is specific to the continuous case. Let $$\mathcal{H}_3(\mathbb{R})= \left\{\left.\left(\begin{array}{ccc}1&x&z\\0&1&y\\0&0&1\end{array}\right)\ \right|\ x,y,z\in \mathbb{R}\right\}$$ be the 3-dimensional real Heisenberg group. Note that $\mathcal{H}_3(\mathbb{R})$ is a non-abelian topological group (in fact a Lie group) with respect to matrix multiplication. Let 
$$A = \left\{\left.\left(\begin{array}{ccc}1&0&z\\0&1&0\\0&0&1\end{array}\right)\ \right|\ z\in \mathbb{R}\right\}$$ be the center of $\mathcal{H}_3(\mathbb{R})$. Then $A \cong \mathbb{R}$ and $\mathcal{H}_3(\mathbb{R})/A \cong \mathbb{R}^2$ as a topological group. This gives an extension of topological groups 
\begin{equation}\label{eqn8}
\begin{split}
0 \to \mathbb{R} \to \mathcal{H}_3(\mathbb{R}) \to \mathbb{R}^2 \to 0.
\end{split}
\end{equation}
The extension is non-split as an extension of topological groups as this would make the group $\mathcal{H}_3(\mathbb{R})$ to be abelian. However, the extension  admits a section $s: \mathbb{R}^2 \to \mathcal{H}_3(\mathbb{R})$ given by $$s(x,y)= \left(\begin{array}{ccc}1&x&\frac{xy}{2}\\0&1&y\\0&0&1\end{array}\right)$$ which is continuous and symmetric, since $$s(-x,-y)= \left(\begin{array}{ccc}1&-x&\frac{xy}{2}\\0&1&-y\\0&0&1\end{array}\right)= s(x,y)^{-1}.$$ Therefore, \eqref{eqn8} represents a unique non-trivial element in $HS_{c}^2(\mathbb{R}^2,  \mathbb{R})$.
\end{example}

\begin{example}\label{example4.3}
We give an example of a topological group and a topological module for which $HS_c^2(-,-) \cong H_c^2(-,-)$.

Examples of (abelian) topological group extensions admitting a symmetric continuous section are guaranteed by a well known result of Michael \cite[Proposition 7.2]{Michael}, which states that: If $X$ and $Y$ are real or complex Banach spaces regarded as topological groups with respect to their addition and $\pi:X \to Y$ is a surjective continuous linear transformation, then there exists a continuous map $s:Y \to X$ such that $ \pi s(y)=y$ and $s(-y)=-s(y)$ for all $y \in Y$.

Let $Y$ and $A$ be infinite dimensional real or complex Banach spaces. Consider $A$ as a trivial $Y$-module. If $0 \to A \to X \to Y \to 0$ is any 
extension of Banach spaces regarded as an extension of topological groups, then the map $X \to Y$ always admits a symmetric continuous section by the above mentioned result of Michael. This together with Proposition \ref{proposition3.2} shows that $HS_c^2(Y,A) \cong H_c^2(Y,A)$.
\end{example}

\begin{example}\label{example4.4}
Let $G$ be a topological group and $A$ be a topological $G$-module. As announced in the previous section, we give an example to show that the homomorphism $$r^*:HS^2_{c}(G,A) \to HS^2(G,A)$$ is not injective in general.

Let $X$ be an infinite dimensional real or complex Banach space and let $A$ be a non-complemented subspace of $X$. The quotient map $X \to X/A$ admits a symmetric continuous section by the above mentioned result of Michael. Since $A$ is non-complemented in $X$, the extension is non-split as an extension of topological groups. But, $X$ is isomorphic to $A \times X/A$ as an abelian group and the extension is split as an extension of abstract groups. Hence $r^*: HS^2_{c}(X/A,A) \to HS^2(X/A,A)$ is not injective.

Let us consider a particular example. Let $k$ be the field of real or complex numbers and $\ell^{\infty}$ be the space of all bounded sequences $x=(x_n)_{n=1}^{\infty}$, where $x_n \in k$ for each $n \geq 1$. Note that $\ell^{\infty}$ is a Banach space with respect to the norm $\|x\|_\infty = \sup_n |x_n|$. Let $c_0$ be the subspace of $\ell^{\infty}$ consisting of all sequences whose limit is zero. This is a closed subspace of $\ell^{\infty}$ and hence a Banach space. By a well-known result due to Phillips \cite[p.33, Corollary 4]{Day}, $c_0$ is a non-complemented subspace of $\ell^{\infty}$ and hence $$0 \to c_0 \to \ell^{\infty} \to \ell^{\infty}/{c_0} \to 0$$ is a non-split extension of topological groups. We would like to mention that there is a general method of constructing non-split extensions of Banach spaces due to Kalton and Peck \cite{Kalton}.
\end{example}

\begin{example}\label{example4.5}
We now give an example to show that the restriction homomorphism $r^*$ is not surjective in general.

Consider the 3-dimensional real Heisenberg group $\mathcal{H}_3(\mathbb{R})$ as an abstract group. Consider the center $A \cong \mathbb{R}$ of $\mathcal{H}_3(\mathbb{R})$ as a topological group with the discrete topology and consider $\mathcal{H}_3(\mathbb{R})/A \cong \mathbb{R}^2$ as a topological group with the usual topology. Then regarding 
\begin{equation}\label{eqn9}
\begin{split}
0 \to \mathbb{R} \to \mathcal{H}_3(\mathbb{R}) \to \mathbb{R}^2 \to 0
\end{split}
\end{equation}
as an extension of abstract groups, we see that it is non-split and admits a symmetric section $s:\mathbb{R}^2 \to \mathcal{H}_3(\mathbb{R})$. Thus, the extension \eqref{eqn9} represents a non-trivial element in $HS^2(\mathbb{R}^2, \mathbb{R})$.

Suppose that, there is a topology on $\mathcal{H}_3(\mathbb{R})$ making \eqref{eqn9} into an extension of topological groups admitting a symmetric continuous section and inducing the underlying abstract group extension. Then $\mathcal{H}_3(\mathbb{R})$ is a topological group with the product topology $\mathbb{R}^2 \times \mathbb{R}$. In particular, the map $I:\mathcal{H}_3(\mathbb{R}) \to \mathcal{H}_3(\mathbb{R})$ sending each matrix to its inverse must be continuous. But this is not true. Consider the open set $$U=\left\{\left.\left(\begin{array}{ccc}1&x&0\\0&1&y\\0&0&1\end{array}\right)\ \right|\ x,y \in (-1,1) \right\}$$ in $\mathcal{H}_3(\mathbb{R})$. Then 
\begin{equation*}
\begin{split}
I^{-1}(U) & =  \left\{\left.\left(\begin{array}{ccc}1&x&z\\0&1&y\\0&0&1\end{array}\right)\ \right|\  \left(\begin{array}{ccc}1&-x&xy-z\\0&1&-y\\0&0&1\end{array}\right)\  \in U  \right\}\\
& = \left\{\left.\left(\begin{array}{ccc}1&x&z\\0&1&y\\0&0&1\end{array}\right)\ \right|\ -x,-y \in (-1,1)~\textrm{and}~xy=z \right\}\\
& = \left\{\left.\left(\begin{array}{ccc}1&x&xy\\0&1&y\\0&0&1\end{array}\right)\ \right|\ x,y \in (-1,1) \right\}\\
\end{split}
\end{equation*}
is not open in $\mathcal{H}_3(\mathbb{R})$. Hence, the element represented by \eqref{eqn9} in $HS^2(\mathbb{R}^2, \mathbb{R})$ has no pre-image in $HS_c^2(\mathbb{R}^2, \mathbb{R})$.
\end{example}
\bigskip

\section{Properties of symmetric continuous cohomology}\label{section5}
Let $G$ and $G'$ be topological groups. Let $A$ be a topological $G$-module and $A'$ a topological $G'$-module. We say that a pair $(\alpha, \beta)$ of continuous group homomorphisms $\alpha: G' \to G$ and $\beta:A  \to A'$ is compatible if the following diagram commutes
$$\begin{CD}
 G@.\times    @.   A   @>>> A\\
 @A{\alpha}AA   @. @VV{\beta}V @VV{\beta}V\\
  {G'}@.\times@. {A'} @>>>{A'}.
\end{CD}$$
In other words, $g'\beta(a) = \beta(\alpha(g')a)$ for all $a \in A$ and $g' \in G$. For simplicity, we write $\psi=(\alpha, \beta)$. Under these conditions, we have the following proposition.

\begin{proposition}\label{proposition5.1}
There is a homomorphism of cohomology groups $\psi^n:HS^n_{c}(G,A) \to HS^n_{c}(G', A')$ for each $n \geq 0$.
\end{proposition}
\begin{proof}
Fix $n \geq 0$. For each $\sigma \in CS_{c}^n(G,A)$, define $\sigma':{G'}^n \to A'$ by $$\sigma'(g'_1,g'_2,g'_3,...,g'_n)= \beta\big( \sigma(\alpha(g'_1),\alpha(g'_2), \alpha(g'_3),...,\alpha(g'_n)) \big)~ \textrm{for all}~(g'_1,g'_2,g'_3,...,g'_n)\in {G'}^n.$$ Clearly $\sigma'$ is continuous being composite of continuous maps. Next we show that $\sigma'$ is symmetric. We have
\begin{eqnarray}
&  & \tau_1\sigma'(g'_1,g'_2,g'_3,...,g'_n)\nonumber\\
& = & -g'_1 \sigma'\big( {g'_1}^{-1},g'_1g'_2,g'_3,...,g'_n \big)\nonumber\\
& = & -g'_1 \beta\big( \sigma(\alpha({g'_1}^{-1}),\alpha(g'_1g'_2), \alpha(g'_3),...,\alpha(g'_n)) \big)\nonumber\\
& = & \beta\big (-\alpha(g'_1) \sigma(\alpha(g'_1)^{-1},\alpha(g'_1)\alpha(g'_2), \alpha(g'_3),...,\alpha(g'_n)) \big)~\textrm{by compatibility}\nonumber\\
& = & \beta\big (\sigma(\alpha(g'_1),\alpha(g'_2), \alpha(g'_3),...,\alpha(g'_n)) \big)~\textrm{by}~ \tau_1\sigma=\sigma\nonumber\\
& = & \sigma'(g'_1,g'_2,g'_3,...,g'_n).\nonumber
\end{eqnarray}
For $2 \leq i \leq n-1$, we have
\begin{eqnarray}
&  & \tau_i\sigma'(g'_1,g'_2,g'_3,...,g'_n)\nonumber\\
& = & -\sigma'\big( g'_1,...,g'_{i-2},g'_{i-1}g'_i, {g'_i}^{-1},g'_ig'_{i+1}, g'_{i+2},...,g'_n \big)\nonumber\\
& = & -\beta\big(\sigma(\alpha(g'_1),...,\alpha(g'_{i-2}),\alpha(g'_{i-1}g'_i), \alpha({g'_i}^{-1}),\alpha(g'_ig'_{i+1}),\alpha(g'_{i+2}),...,\alpha(g'_n) ) \big)\nonumber\\
& = & \beta\big(-\sigma(\alpha(g'_1),...,\alpha(g'_{i-2}),\alpha(g'_{i-1})\alpha(g'_i), \alpha(g'_i)^{-1},\alpha(g'_i)\alpha(g'_{i+1}),\alpha(g'_{i+2}),...,\alpha(g'_n) ) \big)\nonumber\\
& = & \beta\big (\sigma(\alpha(g'_1),\alpha(g'_2), \alpha(g'_3),...,\alpha(g'_n)) \big)~\textrm{by}~ \tau_i\sigma=\sigma\nonumber\\
& = & \sigma'(g'_1,g'_2,g'_3,...,g'_n).\nonumber
\end{eqnarray}
Similarly, we can see that $\tau_n\sigma'(g'_1,g'_2,g'_3,...,g'_n)= \sigma'(g'_1,g'_2,g'_3,...,g'_n)$. This shows that $\sigma'$ is symmetric, and hence an element of $CS_{c}^n(G',A')$.

Define $\psi^n:CS^n_{c}(G,A) \to CS^n_{c}(G', A')$ by $\psi^n(\sigma)=\sigma'$. It it routine to check that $\psi^n$ is a homomorphism commuting with the coboundary operators, that is, the following diagram commutes
$$
\xymatrix{
CS^n_{c}(G,A) \ar[d]^{\partial^n} \ar[r]^{\psi^n} & CS^n_{c}(G', A') \ar[d]^{\partial^n}\\
CS^{n+1}_{c}(G,A) \ar[r]^{\psi^{n+1}} & CS^{n+1}_{c}(G', A').}
$$

This shows that $\psi^n$ preserves both cycles and boundaries and hence defines a map $$\psi^n:HS^n_{c}(G,A) \to HS^n_{c}(G', A')$$ given by $$\psi^n([\sigma])=[\sigma'].$$ It is again routine to check that $\psi^n$ is a homomorphism. This completes the proof.
\end{proof}

In particular, for $G=G'$ and $\alpha=id_G$, we have a homomorphism $\beta^*: HS^*_{c}(G,A)\to HS^*_{c}(G,A')$. The following is an immediate consequence of Proposition \ref{proposition5.1}.

\begin{corollary}
The following statements hold:
\begin{enumerate}
\item Let $H$ be a subgroup of $G$. Then the compatible pair of homomorphisms, the inclusion map $H \hookrightarrow G$ and the identity map $A \to A$, gives the restriction homomorphism $HS^*_{c}(G,A)\to HS^*_{c}(H,A)$.
\item Let $H$ be a normal subgroup of $G$. Then the compatible pair of homomorphisms, the quotient map $G \to G/H$ and the inclusion map $A^H \hookrightarrow A$, gives the inflation homomorphism $HS^*_{c}(G/H,A^H)\to HS^*_{c}(G,A)$.
\end{enumerate}
\end{corollary}

We also have, a long exact sequence in cohomology associated to a short exact sequence of topological $G$-modules admitting a symmetric continuous section which is compatible with the actions.

\begin{proposition}
Let $0 \to A' \stackrel{i}{\to} A \stackrel{j}{\to} A'' \to 0$ be a short exact sequence of topological $G$-modules admitting a symmetric continuous section which is compatible with the actions. Then there is a long exact sequence of symmetric continuous cohomology groups,
$$\cdots \to HS^n_{c}(G,A') \stackrel{i^n}{\to}  HS^n_{c}(G,A)\stackrel{j^n}{\to} HS^n_{c}(G,A'') \overset\delta{\to} HS^{n+1}_{c}(G,A')\to\cdots.$$
\end{proposition}
\begin{proof}
We first show that, for each $n \geq 0$, there is a short exact sequence of symmetric continuous cochain groups
$$0 \to CS^n_{c}(G,A') \stackrel{i^n}{\to}  CS^n_{c}(G,A) \stackrel{j^n}{\to} CS^n_{c}(G,A'') \to 0.$$ Let $\sigma \in CS^n_{c}(G,A')$ be such that $i^n(\sigma)=0$, that is, $i(\sigma(g_1,...,g_n))=0$ for all $(g_1,...,g_n) \in G^n$. But injectivity of $i$ implies that $\sigma(g_1,...,g_n)=0$ for all $(g_1,...,g_n) \in G^n$. Hence $\sigma=0$ and $i^n$ is injective.

Since $ji=0$, we have $j^ni^n=0$ and hence $\Image(i^n) \subseteq \Ker(j^n)$. Suppose $\sigma \in \Ker(j^n)$, that is, $j(\sigma(g_1,...,g_n))=0$. This implies $\sigma(g_1,...,g_n) \in \Ker(j)=\Image(i)$. But $i: A' \to \Image(i)$ is a homeomorphism and hence has a continuous inverse $i^{-1}: \Image(i) \to A'$. Taking $\mu= i^{-1} \sigma$, we have $i^n(\mu)= \sigma$ and hence $\Ker(j^n)\subseteq \Image(i^n)$.

Next we show that $j^n$ is surjective. Let $\sigma \in CS^n_{c}(G,A'')$ and let $s: A'' \to A$ be a symmetric continuous section which is compatible with the actions. Taking $\mu= s \sigma$, we see that $\mu$ is continuous and $j^n(\mu)= j(s \sigma)= \sigma$. It remains to check that $\mu$ is symmetric. We have
\begin{eqnarray}
\tau_1\mu(g_1,g_2,g_3,...,g_n) & = & -g_1 \mu (g_1^{-1},g_1g_2,g_3,...,g_n)\nonumber\\
& = & -g_1 s \big( \sigma(g_1^{-1},g_1g_2,g_3,...,g_n) \big)\nonumber\\
& = & -s \big (g_1 \sigma(g_1^{-1},g_1g_2,g_3,...,g_n) \big)~\textrm{by compatibility of}~ s\nonumber\\
& = & -s \big (-(-g_1 \sigma(g_1^{-1},g_1g_2,g_3,...,g_n)) \big)\nonumber\\
& = & -s \big (-\sigma(g_1,g_2,g_3,...,g_n) \big)~\textrm{by}~ \tau_1\sigma=\sigma\nonumber\\
& = & s \big (\sigma(g_1,g_2,g_3,...,g_n) \big)~\textrm{by symmetry of}~ s\nonumber\\
& = & \mu(g_1,g_2,g_3,...,g_n).\nonumber
\end{eqnarray}
For $2 \leq i \leq n-1$, we have
\begin{eqnarray}
\tau_i\mu(g_1,g_2,g_3,...,g_n) & = & - \mu (g_1,...,g_{i-2},g_{i-1}g_i,g_i^{-1},g_ig_{i+1},g_{i+2},...,g_n)\nonumber\\
& = & -s \big( \sigma(g_1,...,g_{i-2},g_{i-1}g_i,g_i^{-1},g_ig_{i+1},g_{i+2},...,g_n) \big)\nonumber\\
& = & -s \big(-(- \sigma(g_1,...,g_{i-2},g_{i-1}g_i,g_i^{-1},g_ig_{i+1},g_{i+2},...,g_n)) \big)\nonumber\\
& = & -s \big(-\sigma(g_1,g_2,g_3,...,g_n) \big))~\textrm{by}~ \tau_i\sigma=\sigma\nonumber\\
& = & s \big(\sigma(g_1,g_2,g_3,...,g_n) \big)~\textrm{by symmetry of}~ s\nonumber\\
& = & \mu(g_1,g_2,g_3,...,g_n).\nonumber
\end{eqnarray}
Similarly, one can show that $$\tau_n\mu(g_1,g_2,g_3,...,g_n)=\mu(g_1,g_2,g_3,...,g_n).$$ Hence $j^n$ is surjective. Note that only the surjectivity of $j^n$ depends on the choice of the section.

As in the previous proposition, the maps $i^*$ and $j^*$ commute with the coboundary operators and hence we get the following short exact sequence of symmetric continuous cochain complexes $$0 \to CS^*_{c}(G,A') \stackrel{i^*}{\to}  CS^*_{c}(G,A) \stackrel{j^*}{\to} CS^*_{c}(G,A'') \to 0.$$

It is now routine to obtain the desired long exact sequence of symmetric continuous cohomology groups by a diagram chase. This completes the proof.
\end{proof}
\bigskip

\section{Symmetric continuous cohomology of profinite groups}\label{section6}
Profinite groups form a special class of topological groups. We refer the reader to \cite{Ribes} for basic definitions and results regarding profinite groups. The continuous cohomology of a profinite group with coefficients in a discrete module is well studied and equals the direct limit of the cohomology of finite groups. We prove the following similar result for symmetric continuous cohomology.

\begin{theorem}\label{theorem6.1}
The symmetric continuous cohomology of a profinite group with coefficients in a discrete module equals the direct limit of the symmetric cohomology of finite groups.
\end{theorem}
We now set notations for the proof of Theorem 6.1. Let $G$ be a profinite group and $A$ a discrete $G$-module. Let $\mathcal{U}$ be the set of all open normal subgroups of $G$. It can be proved that for each $U \in \mathcal{U}$, the quotient group $G/U$ is finite. Also, for each $U \in \mathcal{U}$, the group of invariants $$A^U=\{a \in A|~ua=a~\textrm{for all}~ u \in U \}$$ is a $G/U$-module by means of the action $$(gU, a) \mapsto ga~ \textrm{for}~ gU \in G/U~\textrm{and}~ a \in A.$$

For elements $U,V \in \mathcal{U}$, we say that $V \leq U$ if $U$ is a subgroup of $V$. This makes $\mathcal{U}$ a directed poset. For $V \leq U$, there are canonical homomorphisms $$\alpha_{UV}:G/U \to G/V~\textrm{and}~\beta_{VU}:A^V \to A^U$$ which form compatible pairs and gives rise to an inverse system of finite groups $\{G/U\}_{U \in \mathcal{U}}$ and a direct system of abelian groups $\{A^U\}_{U \in \mathcal{U}}$. It is then well known that $$G= \varprojlim G/U~\textrm{and}~ A=\varinjlim A^U.$$

Further, for each $n\geq 0$ and each $V \leq U$, as in Proposition \ref{proposition5.1}, the compatible pair of homomorphisms $(\alpha_{UV},\beta_{VU})$ induces a homomorphism $$\psi_{VU}^n:CS^n(G/V, A^V) \to CS^n(G/U, A^U).$$

Thus, we obtain in a natural way the following direct systems of abelian groups over $\mathcal{U}$: $$\{CS^n(G/U, A^U) \}_{U \in \mathcal{U}}~ \textrm{and}~\{HS^n(G/U, A^U) \}_{U \in \mathcal{U}}.$$ Note that, for each $n\geq 0$, the coboundary operator $$\partial_U^n:CS^n(G/U, A^U) \to C^{n+1}(G/U, A^U)$$ commutes with the bonding maps $\psi_{VU}^n$ and hence gives a coboundary operator $$\partial^n: \varinjlim CS^n(G/U, A^U) \to \varinjlim CS^{n+1}(G/U, A^U)$$ making $\{\varinjlim CS^n(G/U, A^U), \partial^n \}$ into a cochain complex.

To prove Theorem \ref{theorem6.1}, it suffices to prove the following lemma, which is essentially \cite[Lemma 6.5.4]{Ribes}.

\begin{lemma}
Let $G$ be a profinite group and $A$ a discrete $G$-module. Then for each $n \geq 0$, there is an isomorphism $$\varinjlim CS^n(G/U, A^U) \cong CS_c^n(G,A)$$ commuting with the corresponding coboundary operators.
\end{lemma}
\begin{proof}
The proof is same as that of \cite[Lemma 6.5.4]{Ribes} and we outline it briefly for the convenience of the readers. Fix $n \geq0$. For each $U \in \mathcal{U}$, let $$\alpha_U: G \to G/U~\textrm{and}~\beta_U: A^U \to A$$ be the obvious homomorphisms. Note that $\beta_U(\alpha_U(g)a)=\beta_U(ga)=ga= g\beta_U(a)$ for all $g \in G$ and $a \in A^U$. Thus, the pair $(\alpha_U, \beta_U)$ is compatible. Let $$\psi_U^n:CS_c^n(G/U, A^U) \to CS_c^n(G,A)$$ be the homomorphism induced by the compatible pair $(\alpha_U, \beta_U)$ as in Proposition \ref{proposition5.1}. Note that this also commutes with the coboundary operators. Considering both $G/U$ and $A^U$ equipped with discrete topology, we have $CS_c^n(G/U, A^U)=CS^n(G/U, A^U)$. Therefore $$\psi_U^n:CS^n(G/U, A^U) \to CS_c^n(G,A).$$ For elements $U,V \in \mathcal{U}$ with $V \leq U$, by definitions, the following diagram commutes
$$
\xymatrix{
CS^n(G/V, A^V) \ar[d]^{\psi_{VU}^n} \ar[r]^{\psi_V^n} & CS_c^n(G,A)\\
CS^n(G/U, A^U) \ar[ru]^{\psi_U^n}.}
$$
Hence, there is a homomorphism $$\psi^n:\varinjlim CS^n(G/U, A^U) \to CS_c^n(G,A)$$ given by $$\psi^n([\sigma_U])= \psi_U^n(\sigma_U)~\textrm{for}~\sigma_U \in CS^n(G/U, A^U).$$ The proof of the bijectivity of $\psi^n$ is routine as in  \cite[Lemma 6.5.4]{Ribes}. The commutativity of the homomorphisms $\psi^n$ with the coboundary operators is immediate from the definitions and the formula \eqref{eqn1}. This proves the lemma.
\end{proof}

Note that $\varinjlim$ is an exact functor and hence we obtain
\begin{eqnarray}
HS_c^n(G,A) & = & H^n(CS_c^*(G,A))\nonumber\\
& = & H^n(\varinjlim CS^*(G/U, A^U)) \nonumber\\
& = & \varinjlim H^n(CS^*(G/U, A^U))\nonumber\\
& = & \varinjlim HS^n(G/U, A^U).\nonumber
\end{eqnarray}
This completes the proof of the Theorem \ref{theorem6.1}.
\bigskip

\section{Symmetric smooth cohomology of Lie groups}\label{section7}
The theory of Lie groups, particularly cohomology of Lie groups, has been studied from different points of view. Various cohomology theories of Lie groups have been constructed in the literature \cite{Hochschild1, Hochschild2, Hu, vanEst, Mostow}. There is a rich interplay between the continuous cohomology of a Lie group, the cohomology of its Lie algebra and the de Rham cohomology of its associated symmetric space \cite{Bott, Swierczkowski}.

There is a well known theory of smooth cohomology of a Lie group $G$ with coefficients in a topological vector space $V$ on which $G$ acts smoothly. This theory was defined by Blanc \cite{Blanc} and was later extended by Brylinski \cite{Brylinski} to coefficients in an arbitrary abelian Lie group. In this section, we define the symmetric smooth cohomology of a Lie group and prove some basic properties as we did for topological groups.

Let $G$ be a Lie group and $A$ be a smooth $G$-module. We can define an analogous cohomology theory by imposing the condition that the standard cochains are symmetric and smooth. More precisely, for each $n \geq 0$, let $C^n_{s}(G,A)$ be the group of all smooth maps from the product Lie group $G^n \to A$ and let the coboundary be given by the standard formula as in \eqref{eqn1}. Analogous to  the construction in the continuous case, for each $n\geq 0$, consider the action of the symmetric group $\Sigma_{n+1}$ on $C^n_{s}(G,A)$ as given by equations \eqref{eqn2}. The smoothness of the action of $G$ on $A$ implies that the action is well-defined. As in the continuous case, the action is compatible with the standard coboundary operators $\partial^n$ and hence gives the subcomplex of invariants $$\{CS_{s}^n(G, A), \partial^n\}_{n \geq 0} =\{C_{s}^n(G, A)^{\Sigma_{n+1}}, \partial^n\}_{n \geq 0}.$$ We define the symmetric smooth cohomology $HS^n_{s}(G,A)$ to be the cohomology groups of this new cochain complex.

We obtain some basic properties of this cohomology theory as follows.

\begin{proposition} Let $G$ be a Lie group and $A$ be a smooth $G$-module. Then we have the following:
\begin{enumerate}
\item $HS^0_{s}(G,A)=A^G.$
\item $ZS^1_{s}(G,A)=$ the group of symmetric smooth crossed homomorphisms from $G$ to $A$.
\item The map $h^*:HS_{s}^2(G, A) \to H_{s}^2(G, A)$ is injective.
\item Let $A$ be a $G$-module and $A'$ be a $G'$-module such that the actions are compatible. Then there is a homomorphism of cohomology groups $HS^n_{s}(G,A) \to HS^n_{s}(G', A')$ for each $n\geq 0$.
\item Let $0 \to A' \stackrel{i}{\to} A \stackrel{j}{\to} A'' \to 0$ be a short exact sequence of smooth $G$-modules admitting a symmetric smooth section which is compatible with the actions. Then there is a long exact sequence of symmetric smooth cohomology groups
$$\cdots \to HS^n_{s}(G,A') \stackrel{i^n}{\to} HS^n_{s}(G,A) \stackrel{j^n}{\to} HS^n_{s}(G,A'')\overset\delta{\rightarrow} HS^{n+1}_{s}(G,A')\to\cdots.$$
\end{enumerate}
\end{proposition}
\begin{proof} We leave the proofs to the reader as they are similar to those of the continuous case.
\end{proof}

As in the continuous case, we would like to have an interpretation of the symmetric smooth cohomology in dimension two. For that purpose, we recall that, an extension of Lie groups $$0 \to A \stackrel{i}{\to} E \stackrel{\pi}{\to} G \to 1$$ is an algebraic short exact sequence of Lie groups with the additional property that both $i$ and $\pi$ are smooth  homomorphisms and $\pi$ admits a smooth local section $s:U \to E$, where $U \subset G$ is an open neighbourhood of identity. The existence of a smooth local section means that $E$ is a principal $A$-bundle over $G$ with respect to the left action of $A$ on $E$ given by $(a,e) \mapsto i(a)e$ for $a \in A$ and $e \in E$. Since, an extension of Lie groups is a principal bundle, it follows that it is a trivial bundle ($E$ is $A \times G$ as a smooth manifold) if and only if it admits a smooth section.

Two extensions of Lie groups $0 \to A \stackrel{i}{\to} E \stackrel{\pi}{\to} G \to 1$ and $0 \to A \stackrel{i'}{\to} E' \stackrel{\pi'}{\to} G \to 1 $ are said to be equivalent if there exists a smooth isomorphism $\phi:E \to E'$ with smooth inverse such that the following diagram commute
$$
\xymatrix{
0 \ar[r] & A \ar@{=}[d] \ar[r]^{i} & E \ar[r]^{\pi} \ar[d]^{\phi} & G \ar[r] \ar@{=}[d] & 1 \\
0 \ar[r] & A \ar[r]^{i'} & E' \ar[r]^{\pi'} & G \ar[r] & 1.}
$$

Let $\mathcal{S}(G,A)$ denote the set of equivalence classes of Lie group extensions of $G$ by $A$ admitting a symmetric smooth section and corresponding to the given way in which $G$ acts on $A$. Such extensions are classified by the second symmetric smooth cohomology as follows.

\begin{theorem}\label{theorem7.2}
Let $G$ be a Lie group and $A$ be a smooth $G$-module. Then there is a bijection $\Psi:\mathcal{S}(G,A) \to HS^2_{s}(G,A)$
\end{theorem}
\begin{proof}
We leave the proof to the reader as it is similar to that of the continuous case.
\end{proof}

Note that countable groups with the discrete topology are 0-dimensional Lie groups. Taking $G$ to be an abelian Lie group of positive dimension in the Example \ref{example4.1}, we have an example of a positive dimensional Lie group and a smooth module for which the two cohomology theories $HS_s^*(-,-)$ and $H_s^*(-,-)$ are different. Similarly, Example \ref{example4.2} also serves as an example for the Lie group case.

Note that an extension of Lie groups $0 \to A \stackrel{i}{\to} E \stackrel{\pi}{\to} G \to 1$ can be thought of as an extension of topological groups by considering only the underlying topological group structure ($i$ becomes closed continuous and $\pi$ becomes open continuous). This gives the restriction homomorphism $$r^*:HS^n_{s}(G,A) \to HS^n_{c}(G,A)~ \textrm{for each}~ n \geq 0.$$ We investigate this homomorphism in dimension two. Before that, we recall Hilbert's fifth problem, which asked: Is every locally Euclidean topological group necessarily a Lie group? It is well known that Hilbert's fifth problem has a positive solution \cite{Gleason, Montgomery, Yamabe}. We use this in the following concluding theorem.

\begin{theorem}
Let $G$ be a Lie group and $A$ be a smooth $G$-module. Then the natural homomorphism $$r^*:HS^2_{s}(G,A)\to HS^2_{c}(G,A)$$ is an isomorphism.
\end{theorem}
\begin{proof}
Let $[\sigma] \in HS^2_{s}(G,A)$ and let $0 \to A \to E \to G \to 1$ be an extension of Lie groups corresponding to $[\sigma]$ by Theorem \ref{theorem7.2}, which is unique up to equivalence of extensions. Suppose that $r^*([\sigma])$ is trivial in $HS^2_{c}(G,A)$. Then there exists a continuous section $s: G \to E$ which is a group homomorphism. This gives a continuous isomorphism between the Lie groups $E$ and $A \rtimes G$. A continuous homomorphism between Lie groups is smooth \cite[Theorem 4.21]{Michor}. As a consequence this isomorphism is smooth. Hence the cohomology class $[\sigma]$ is trivial in $HS^2_{s}(G,A)$ and the homomorphism $r^*$ is injective.

Let $[\sigma] \in HS^2_{c}(G,A)$ and let $0 \to A \stackrel{i}{\to} E_{\sigma} \stackrel{\pi}{\to} G \to 1$ be the extension of topological groups defined using the 2-cocycle $\sigma$ as in the proof of Theorem \ref{theorem3.3}. The extension admits a symmetric continuous section $s:G \to E_{\sigma}$ given by $s(g)=(0,g)$ for all $g \in G$. By construction, $E_{\sigma}$ is $A \times G$ as a topological space. Also, $E_{\sigma}$ is locally Euclidean as both $A$ and $G$ are Lie groups. Hence, we conclude that $E_{\sigma}$ is a Lie group by the positive solution to Hilbert's fifth problem. Since a continuous homomorphism between Lie groups is smooth, we have that both $i$ and $\pi$ are smooth homomorphisms. Applying the implicit function theorem, we can find a smooth section of $\pi$ defined in a neighbourhood of identity in $G$. This shows that $0 \to A \to E_{\sigma} \to G \to 1$ is an extension of Lie groups. Further, the section $s:G \to E_{\sigma}$ becomes smooth as $E_{\sigma}$ has the product smooth structure. Let $\sigma'$ be the symmetric smooth 2-cocycle associated to the symmetric smooth section $s$. Then $r^*([\sigma'])=[\sigma]$ and the homomorphism $r^*$ is surjective.
\end{proof}

\bibliographystyle{amsplain}

\end{document}